\documentclass[letterpaper,notitlepage,11pt,twoside]{article} 

\usepackage{amsmath} 
\usepackage{amssymb} 
\usepackage{amsthm} 
\usepackage{bbm} 

\numberwithin{equation}{section}
\usepackage[backend=bibtex,maxnames=15]{biblatex} 
\bibliography{biblio2}
\AtEveryBibitem{%
    \clearfield{url}
    \clearfield{urldate}%
    \clearfield{doi}%
		\clearfield{eprint}%
		\clearfield{ISSN}
}
\renewbibmacro{in:}{}
\usepackage{graphicx} 
\usepackage{subcaption} 
\usepackage{enumerate} 
\usepackage{enumitem}
\setlist[enumerate]{itemsep = -0.2em}

\usepackage{authblk} 

\usepackage[explicit]{titlesec}  
\titleformat{\subsection}[runin]
  {\normalfont\normalsize\bfseries}{\thesubsection}{0.3em}{#1.}
	
\titleformat{\section}
  {\normalfont\large}{\thesection}{0.2em.}{\centering \textsc{#1}}

\usepackage{caption}  
\usepackage{hyperref} 
\usepackage{color} 
\definecolor{MyDarkBlue}{rgb}{0,0.08,0.50}  
\definecolor{BrickRed}{rgb}{0.65,0.08,0}

\hypersetup{  
colorlinks=true,       
    linkcolor=MyDarkBlue,          
    citecolor=BrickRed,        
    filecolor=red,      
    urlcolor=cyan           
}

\usepackage[headheight=0.6cm,headsep=0.6cm,margin=1in,bottom=1.1cm]{geometry} 

\usepackage{fancyhdr}
\newtheorem{Lemma}{Lemma}[section]

\newtheorem{Proposition}[Lemma]{Proposition}

\newtheorem{Theorem}[Lemma]{Theorem}

\newtheorem{Remark}[Lemma]{Remark}

\newtheorem{Corollary}[Lemma]{Corollary}
\newtheorem{Definition}[Lemma]{Definition}
\newtheorem{Condition}[Lemma]{Condition}

\newcommand{\sub}[1]{\boldsymbol{#1}}
\newcommand{\R}{\mathbb{R}}

\newcommand{\N}{\mathbb{N}}
\newcommand{\I}{\mathbbm{1}}
\newcommand{\pr}{\mathbb{P}}
\newcommand{\E}{\mathbb{E}}

\newcommand{\G}{\mathcal{G}}

\newcommand{\eqn}[1]{\begin{equation} #1 \end{equation}}

\newcommand{\e}{\mathrm{e}}
\newcommand{\CBP}{\mathrm{CBP}}
\newcommand{\La}{\mathcal{L}}
\newcommand{\CTBP}{\mathrm{CTBP}}
\newcommand{\CTBPs}{\mathrm{CTBPs}}

\newcommand{\F}{\mathcal{F}}
\newcommand{\Din}{D^{\scriptscriptstyle(\mathrm{in})}}

\newcommand{\sss}{\scriptscriptstyle}

\begin{document}

\title{\bfseries\uppercase{\large from trees to graphs: collapsing \\ continuous-time branching processes} }


\author[a,1]{Alessandro Garavaglia}
\author[a,2]{Remco van der Hofstad}
\affil[a]{\footnotesize Department of Mathematics and
    Computer Science, Eindhoven University of Technology, 5600 MB Eindhoven, The Netherlands}

\vspace{0.2cm}
\affil[$ $]{{\itshape email address}: $^1$a.garavaglia@tue.nl, $^2$rhofstad@win.tue.nl}

\date{}

\maketitle

\vspace{-1cm}
\begin{abstract}
Continuous-time branching processes ($\CTBPs$) are powerful tools in random graph theory, but are not appropriate to describe real-world networks, since they produce trees rather than (multi)graphs. In this paper we analyze collapsed branching processes (CBPs), obtained by a collapsing procedure on $\CTBPs$, in order to define multigraphs where vertices have fixed out-degree $m\geq 2$. A key example consists of preferential attachment models (PAMs), as well as generalized PAMs where vertices are chosen according to their degree and age.
We identify the degree distribution of $\CBP\mathrm{s}$, showing that it is closely related to the limiting distribution of the $\CTBP$ before collapsing. In particular, this is the first time that $\CTBPs$ are used to investigate the degree distribution of PAMs beyond the tree setting.
\end{abstract}

\thispagestyle{plain}

\section{Introduction and main results}
\label{sec-intr}

\subsection{Our model and main result}
\label{sec-model}
The main result of this paper is the definition of multigraphs from continuous-time branching processes ($\CTBP$), through a procedure that we call {\em collapsing}. We analyze the case where we collapse a fixed number $m\in\N$ of individuals. The heuristic idea is to consider the tree defined by the branching process, and collapse or merge together $m$ different nodes in the tree to create a vertex in the multigraph. Throughout this paper, we will consider an {\itshape individual} to be a node in the tree of the branching process, while a {\itshape vertex} is a node in the multigraph  obtained by collapsing.

We recall now some notation on $\CTBPs$. For a more detailed introduction, we refer to Section \ref{sec-bptheory}. We consider a branching process $\sub{\xi}$ defined by a birth process $(\xi_t)_{t\geq 0}$. In these models, individuals produce children according to i.i.d. copies of the process $(\xi_t)_{t\geq 0}$. Usually, individuals in the branching populations are denoted by $x = \emptyset x_1\cdots x_k$ (see Definition \ref{def-BP}). In this paper, we will not denote individuals with their position in the genealogical tree, but rather by their birth order. Denote the sequence of birth times of individuals in the branching population by $(\tau_n)_{n\in\N}$.

Fix $m\in\N$. We denote $(n,j) = m(n-1)+j$, for $j=1,\ldots,m$. We now give the precise definition of the collapsed branching process:

\begin{Definition}[Collapsed branching process]
\label{def-collBP}
Consider a branching process $\sub{\xi}$. Then, a {\em collapsed branching process} is a random process $(\CBP^{\sss(m)}_t)_{t\geq0}$, for which, for every $t\geq 0$,
$\CBP^{\sss(m)}_t$ is a directed multigraph with adjacency matrix $(g_{x,y}(t))_{x,y\in\N}$, where
\eqn{
	g_{x,y}(t) = \sum_{j=1}^{m}\I_{\{(x,j)\rightarrow (y,1),\ldots,(y,m)\}}\I_{\R^+}(t-\tau_{(x,j)}),
}
and $\{(x,j)\rightarrow(y,1),\cdots,(y,m)\}$ is the event that there is a directed edge between individual $(x,j)$ and one of the individuals  $(y,1),\ldots,(y,m)$ in the tree defined by the branching process at time $t$. We denote the size of $\CBP^{\sss(m)}_t$ by $N^{\sss(m)}(t)$.
\end{Definition}

As the reader can see from the definition, the collapsing procedure combines $m$ individuals together with their edges to create a vertex, and there is an  edge between two vertices if and only if there is an edge between a pair of individuals collapsed to create the two vertices.  $\CBP^{\sss(m)}_t$ is a graph where every vertex (except vertex 1) has out-degree $m$. Self-loops and multiple edges are allowed (see Figure \ref{fig:tree} for an example of CBP).

We consider the birth time of the vertex $n$ in the multigraph to be
$\tau_{(n,1)} = \tau_{m(n-1)+1}$.
Thus, vertex $n$ in $\CBP^{\sss(m)}$ is considered alive when $(n,1)$ is alive in $\sub{\xi}$. Notice that when $n$ is born, it has only one out-edge, because the other individuals $(n,2),\ldots,(n,m)$ are not yet alive. Clearly, the in-degree at time $t$ of a vertex $n$ in $\CBP^{\sss(m)}$ is given by
$$
	D^{\scriptscriptstyle(\mathrm{in})}_n(t) = \sum_{j=1}^m \xi^{(n,j)}_{t-\tau_{(n,j)}}.
$$ 
The main difference between CBPs and Preferential Attachment Models (PAMs) is that CBPs are defined in continuous-time, while time in PAMs is discrete. Heuristically, discrete time in PAMs is described as the time unit at which a nex vertex is added to the graph (see for instance \cite{ABrB}, \cite[Chapter 8]{vdH1}, \cite{Bol01}), while in CBPs time is continuous and new vertices are born at exponential rate (\cite[Theorem A]{RudValko}, \cite[Theorem 5.4]{Nerman},  Theorem \ref{the-BPmain} below).
\begin{figure}[h]
\centering
 \begin{subfigure}[b]{0.55\textwidth}
        \includegraphics[width= \textwidth]{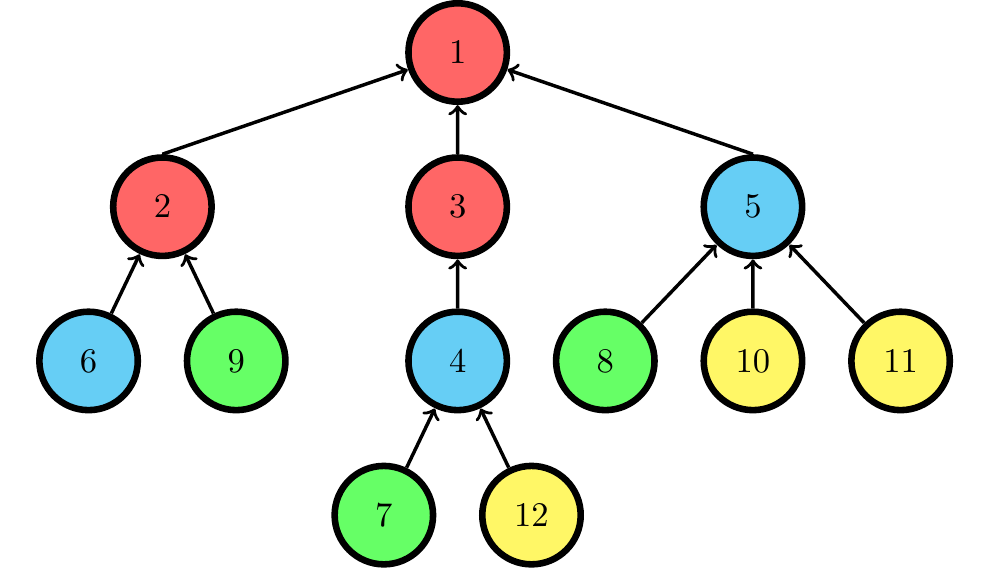}
				\caption{Branching process tree}
    \end{subfigure}
		\hspace{1cm}
		\begin{subfigure}[b]{0.33\textwidth}
        \includegraphics[width=\textwidth]{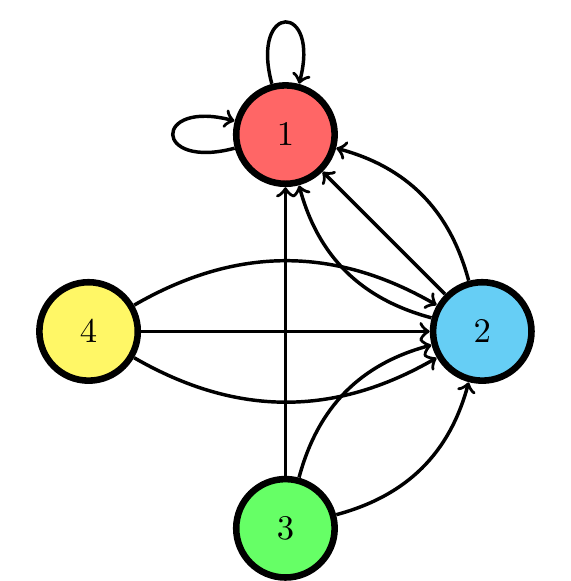}
        \caption{Collapsed branching process}
    \end{subfigure}
		\caption{An example of a collapsed branching process where vertices have fixed out-degree $m=3$.}
		\label{fig:tree}
\end{figure}

\subsection{Results}
\label{sec-results}
 
Our main goal is to show that we can define a multigraph from a CTBP, and analyze its rate of growth as well as the limiting degree distribution. 

Our results are a first attempt to create a link between trees and multigraphs in continuous time. The collapsing procedure creates difficulties though. For instance, we consider different individuals to create a vertex, each one of them having its own birth time. This has to be taken into account to investigate the degree evolution of a vertex in $\CBP$.

Here we state the result on the limiting degree distribution of CBPs, relying on properties of CTBPs as formulated in Theorem \ref{theo-general} below:

\begin{Theorem}[Limiting degree distribution of CBPs]
\label{the-limitCBP}
Consider a branching process $\sub{\xi}$, and fix $m\in\N$. Denote the size of $\CBP_t^{\sss(m)}$ by $N^{\sss(m)}(t)$ and the number of vertices with degree $k$ by $N^{\sss(m)}_k(t)$. Under the hypotheses of Theorem \ref{theo-general}, as $t\rightarrow\infty$, 
\eqn{
\label{for:pmkgeneral}
	\frac{N^{\sss(m)}_k(t)}{N^{\sss(m)}(t)}\stackrel{\pr}{\displaystyle \longrightarrow}p_k^{\sss(m)}=
		\pr\left(\xi^1_{T_{\alpha^*}}+\cdots+\xi^m_{T_{\alpha^*}}=k\right),
}
where $(\xi^1_t)_{t\geq0},\ldots,(\xi_t^m)_{t\geq0}$ are $m$ independent copies of the birth process $(\xi_t)_{t\geq0}$, $\alpha^*$ is the Malthusian parameter of $\sub{\xi}$, and $T_{\alpha^*}$ is an exponentially distributed random variable with parameter $\alpha^*$.
\end{Theorem}
The hypotheses of Theorem \ref{theo-general} are technical, and they are deferred to later. Theorem \ref{the-limitCBP} is part of Theorem \ref{theo-general}, that is more general and requires notation from CTBPs theory that we introduce in Section \ref{sec-bptheory}. 

\subsection{Embedding PAMs}
In discrete time, PAMs are defined as a sequence of random graphs $(\G_n)_{n\in\N}$, where at every step a new vertex is introduced in the graph. In general, the attachment rule is given in terms of a function of the degree $f$ that we call the  PA {\em function} or {\em weight}. Conditionally on the graph $\G_{(n,j)}$ where the $j$-th edge of the $n$-th vertex has been added, 
\eqn{
\label{for-PAtransition}
	\pr\left(n\stackrel{j+1}{\rightarrow}i~|~\G_{(n,j)}\right) = \frac{f(D_i(n,j))}{\sum_{h=1}^nf(D_h(n,j))},
}
where $D_i(n,j)$ denotes the degree of the vertex $i$ in $\G_{(n,j)}$. When $f$ is affine, it is possible to define the model with out-degree $m\geq 2$ from the tree case where the out-degree is 1 (we refer to \cite[Chapter 8, Section 8.2]{vdH1} for the precise definition). In particular, the collapsing procedure we introduced in Definition \ref{def-collBP} mimics the construction of PAMs with affine attachment function.

Several works in the literature (\cite{Athr}, \cite{Athr2}, \cite{RudValko}) use CTBPs to investigate the degree distribution of PA trees. In particular, embedding theorems are proved between discrete and continuous time (see \cite[Theorem 3.3]{Athr}, \cite[Theorem 2.1]{Athr2}). These results are based on the fact that all intervals between two jumps in every copy of the birth process 
$(\xi_t)_{t\geq0}$ are exponentially distributed. This means that, conditionally on the present state of the tree, the probability that a new vertex is attached to the $i$-th vertex already present is just the ratio between the PA function of the degree of vertex $i$ and the total weight of the tree.
Also PAMs with out-degree $m\geq2$ have been investigated, but not through embeddings of CTBPs.

It is possible to construct a CBP that embeds PAMs with affine attachment function. We need to define a suitable birth process for this:
\begin{Definition}[Embedding birth process]
\label{def-emb_birthpr}
Consider a sequence of positive numbers $(\lambda_k)_{k\in\N}$. Let $(E_k)_{k\in\N}$ be a sequence of independent and exponentially distributed random variables, with $E_k \sim E(\lambda_k)$, and $E_{-1}=0$. We call $(\xi_t)_{t\geq0}$ the {\em embedding birth process}, where $\xi_t = k$ if $t\in[E_{-1}+\cdots+E_{k-1},E_{-1}+\cdots+E_k)$.
\end{Definition}
This construction in used in \cite{RudValko}, \cite{Athr}, \cite{Athr2}. It allows to embed PA trees in continuous time where the PA function is given by $f(k) = \lambda_k$. Embedding birth processes allow us to describe PAMs with out-degree $m\geq2$ and affine $f$ using CBPs. In fact, an immediate application of \cite[Theorem 3.3]{Athr} and \cite[Theorem 2.1]{Athr2} is enough to prove that the transition probability in CBP from $\CBP_{\tau_{(n,j)}}^{\sss(m)}$ to $\CBP_{\tau_{(n,j+1)}}^{\sss(m)}$ are exactly given by 
\eqref{for-PAtransition}, with the restriction that the first edge of every vertex cannot be a self-loop. In particular, this yields the following result:

\begin{Corollary}[Continuous-time PAM]
\label{cor-PAM}
	Fix $m\geq2$ and $\delta>-m$. Let $(\xi_t)_{k\in\N}$ be an embedding birth process defined by the sequence $(k+1+\delta/m)_{k\in\N}$. Then, the corresponding CBP embeds  the PAM in continuous time with attachment rule $f(k) = k+\delta$, and satisfies Theorem \ref{the-limitCBP} (and Theorem \ref{theo-general}). As a consequence, the limiting degree distribution is given by
\eqn{
\label{for-PAMdeg}
	p_k^{\sss(m)} =  \left(2+\delta/m\right)\frac{\Gamma(2+\delta/m+m+\delta)}{\Gamma(m+\delta)}
						\frac{\Gamma(k+m+\delta)}{\Gamma(k+m+\delta+3+\delta/m)}.
}
\end{Corollary}

Corollary \ref{cor-PAM} is the application of Theorem \ref{the-limitCBP} to the case of the CTBPs that embed PAMs in continuous time. Indeed, the CBP observed at times  $(\tau_n)_{n\in\N}$ ( the sequence of birth times of the CTBP) corresponds to the discrete-time PAM. However, since the ratio $N^{\sss(m)}_k(t)/N^{\sss(m)}(t)$ converges in probability, Theorem \ref{the-limitCBP}  does not imply the convergence along the sequence $(\tau_n)_{n\in\N}$. To prove that the convergence holds also in discrete time, a more detailed analysis is necessary, therefore we state it as a separate result:
\begin{Theorem}[Discrete-time PAMs]
\label{th:PAMdiscrete}
Fix $m\geq2$ and $\delta>-m$. Let $(\xi_t)_{k\in\N}$ be an embedding birth process defined by the sequence $(k+1+\delta/m)_{k\in\N}$. Consider the corresponding discrete-time PAM defined as  $\mathrm{PA}_{n,j}(m,\delta) = \CBP^{\sss(m)}_{\tau_{(n,j)}}$, for $n\in\N$ and $j\in[m]$. Then, for every $k\in\N$, the fraction of vertices with degree $k$ in $\mathrm{PA}_{n,j}(m,\delta)$ converges in probability to $p^{\sss(m)}_k$ as in \eqref{for-PAMdeg}.
\end{Theorem}
While CTBP arguments have been used a lot in the context of PA trees (for which $m=1$),  Theorem \ref{th:PAMdiscrete} provides the first example where it is applied beyond the tree setting.  Thus, our results offer to opportunity to use the powerful CTBP tools in order to study PAMs.

To show the universality of our collapsing construction, we apply Theorem \ref{the-limitCBP} to another classical random graph model. A random recursive tree (RRT) is a sequence of PA trees where the attachment function $f$ is equal to one. At every step, a vertex is added to the tree and attached uniformly to one existing vertex (see \cite{vdH2001} for an introduction). In this case we obtain the following result:
\begin{Corollary}[Random recursive graph]
\label{cor-RRG}
Fix $m\geq2$. Let $(\xi_t)_{k\in\N}$ be an embedding birth process defined by the sequence $\lambda_k = 1$ for every $k\in\N$. Then, the corresponding CBP defines a sequence of random graphs which transition probabilities are given by
\eqn{
\label{for-transprob-1}
	\pr\left(n\stackrel{j+1}{\rightarrow}i ~|~ \CBP^{\sss(m)}_{\tau_{(n,j)}}\right) =\left\{
	\begin{array}{lcr}
		\frac{1}{(n-1)+j/m} & 
			& \mbox{if}~i\neq n,\\
						& & \\
	\frac{j/m}{(n-1)+j/m} & & \mbox{if}~i= n.\\
		\end{array}\right.
}
We call the sequence of random graphs defined by \eqref{for-transprob-1} {\em random recursive graph}. As a consequence, the limiting degree distribution is given by
\eqn{
\label{for-RRGdistr}
	p^{\sss(m)}_k = \frac{1}{m+1}\left(1+\frac{1}{m}\right)^{-k}.
}
Consequently, the same result also holds in discrete time.
\end{Corollary}
In this case the CBP can be seen as the generalization of the RRT to the case where the out-degree is $m\geq2$. In particular, when $m=1$ the distribution in \eqref{for-RRGdistr} reduces to $p^{\sss(1)}_k = 2^{-(k+1)}$, which is the known limiting degree distribution for the RRT (see \cite{Janson}).

An extension of the PAM has been proposed by us in \cite{GarvdHW}, where we introduce fitness and aging in preferential attachment trees. The methodology used in the present work is applicable to the case with aging only. The fitness case is not tractable, and we explain the reason in Section \ref{sec-discussion}. A preferential attachment tree with aging is given in terms of a  CTBP, where we introduce the effect of aging, i.e., the probability of generating a child decreases with age. For a precise definition of such processes, we refer to \cite[Section 2.2]{GarvdHW}.

\begin{Definition}[Aging birth process]
\label{def-agingprocess}
Consider a sequence of positive numbers $(\lambda_k)_{k\in\N}$, and the corresponding embedding birth process as in Definition 
\ref{def-emb_birthpr}. Consider a function $g:\R^+\rightarrow\R^+$ called {\em aging function}, such that $\int_0^\infty g(t)dt<\infty$.  Defining $G(t) = \int_0^tg(s)ds$, we call $(\xi_{G(t)})_{t\geq0}$ an {\em aging birth process}. 
\end{Definition}

The assumption on the integrability of $g$ is not necessary, but as shown in  \cite{GarvdHW} this is the non-trivial  case of aging effect.
In  \cite{GarvdHW} we prove that a CTBP defined by an aging birth process has a limiting degree distribution $(p^{\sss(1)}_k)_{k\in\N}$ with exponential tail, under the condition $\lim_{t\rightarrow\infty}\E[\xi_{G(t)}]>1$.  The result of the present paper can also be applied to the aging birth processes, leading to the following result:
\begin{Corollary}[Aging PAMs]
\label{cor-aging}
Fix $m\geq 2$, $\delta>-m$, and define the sequence $(k+1+\delta/m)_{k\in\N}$. Denote the corresponding embedding birth process by $(\xi_t)_{t\geq0}$. Let $g$ be an aging function as in Definition \ref{def-agingprocess}, such that $g(t)\leq \bar{g}$ for some constant $\bar{g}>0$ and for every $t\geq0$. Assume that $\lim_{t\rightarrow\infty}\E[\xi_{G(t)}]>1$. Then, the CBP obtained by the CTBP defined by the aging process satisfies Theorem \ref{the-limitCBP} (and Theorem \ref{theo-general}). As a consequence, the limiting degree distribution $(p^{\sss(m)}_k)_{k\in\N}$ satisfies 
\eqn{
\label{for-distrAging-CBP}
	p^{\sss(m)}_k = \frac{\Gamma(k+m+\delta)}{\Gamma(k+1)}\e^{-Ck}(1+o(1)),
}
where $C = |\log(1-\e^{-\int_0^\infty g(t)dt})|$.
\end{Corollary}
In particular, it is possible to show that the transition probabilities of the discrete-time version $(\CBP^{\sss(m)}_{\tau_{(n,j)}})_{n\in\N,j\in[m]}$ of a CBP defined by an aging process satisfies
\eqn{
\label{for-agePAM}
	\pr\left(n\stackrel{j+1}{\rightarrow}i~|~\CBP^{\sss(m)}_{\tau_{(n,j)}},\tau_{(n,j+1)}\right) \approx \frac{(D_i(\tau_{(n,j)})+\delta)g(\tau_{(n,j+1)}-\tau_{(i,1)})}{\sum_{h=1}^n(D_h(\tau_{(n,j)})+\delta)g(\tau_{(n,j+1)}-\tau_{(h,1)})},
}
where $D_i(t)$ denotes the total degree of vertex $i$ in $\CBP_t^{\sss(m)}$ and the approximation is due to the fact that we consider $\tau_{(i,1)}$ as the birth time of all the $m$ individuals collapsed to generate vertex $i$. The expression in \eqref{for-agePAM} for the attachment rule in the presence of aging resembles the ones given in other works about aging in PAMs (\cite{WuFu}, \cite{WaMiYu}, \cite{WanGen}).

\subsection{Discussion and open problems}
\label{sec-discussion}

\medskip
\paragraph{\bf Neighborhoods in CBP.} $\CBP$ with fixed out-degree $m\geq 2$ is a continuous-time random graph model of which the size of the graph grows exponentially in time (see \eqref{for:mainT-1}), and we are able to describe its limiting degree distribution. In particular, we can see a branching tree as a special case of $\CBP$ with $m=1$. This is an attempt to translate properties from $\CTBP$ to multigraphs. As a consequence, we might ask what other topological properties a $\CBP$ inherits from the underlying $\CTBP$. As an example, PAMs are known to be {\em locally tree-like} graphs (see \cite{Borgs14}), prompting the question whether this is true because PAMs can be defined as CBPs. 

\begin{figure}[!h]
\centering
 \begin{subfigure}[b]{0.4\textwidth}
				\centering
        \includegraphics[width= 0.6\textwidth]{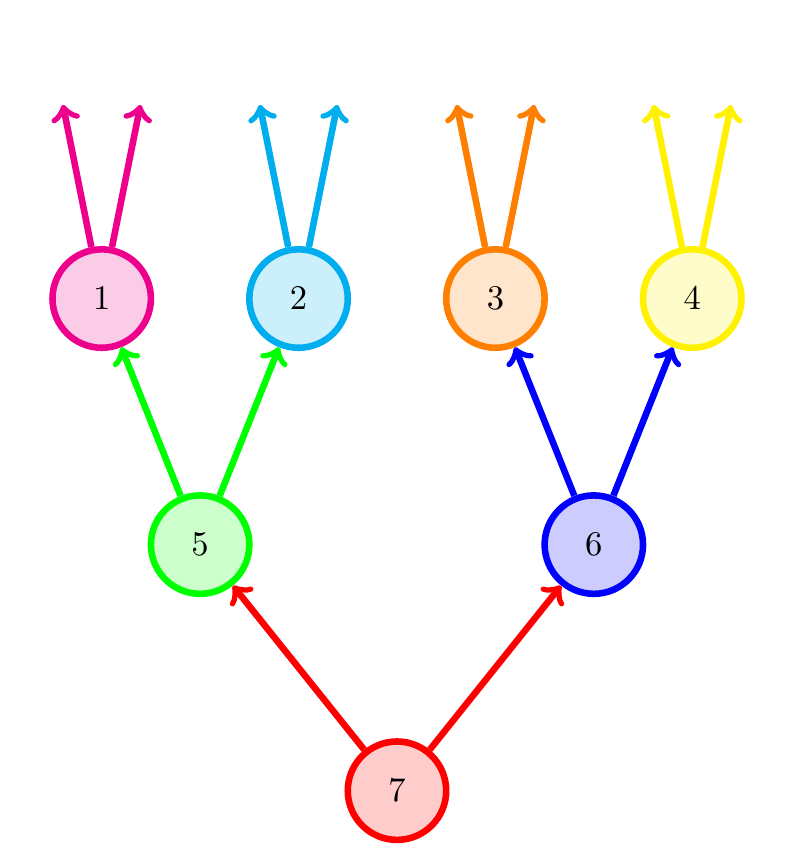}
				\caption{$\CBP$}
    \end{subfigure}
		\begin{subfigure}[b]{0.4\textwidth}
				\centering
        \includegraphics[width= 0.7\textwidth]{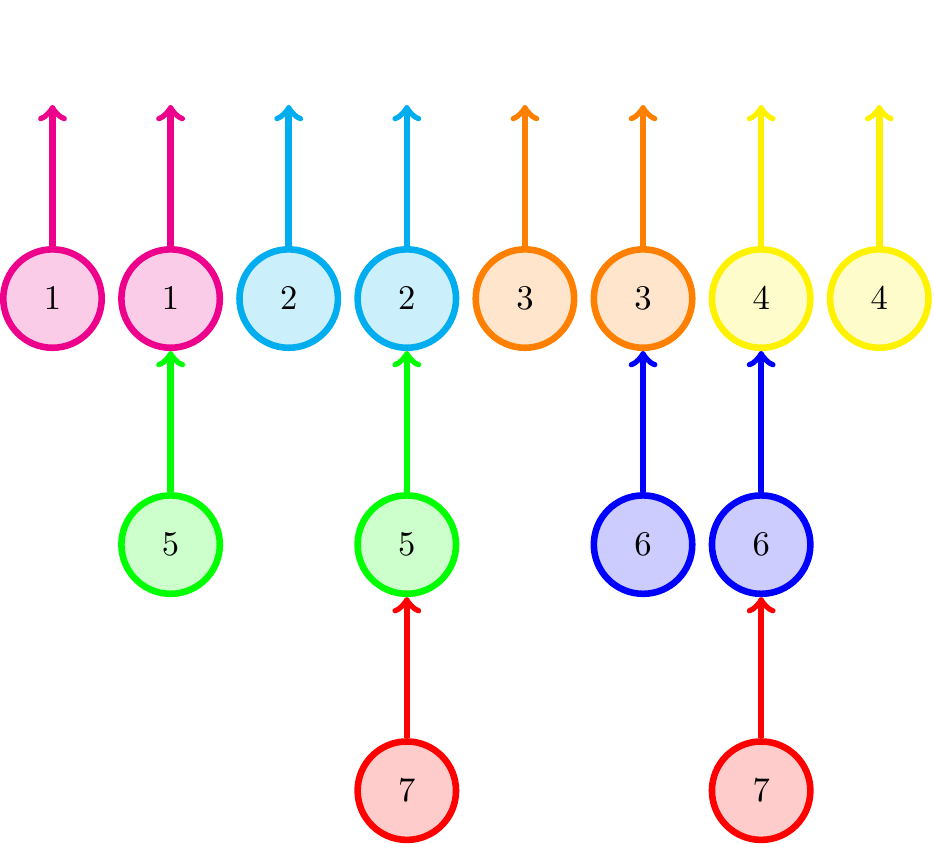}
        \caption{$\CTBP$}
    \end{subfigure}
		\caption{An example of minimum degree tree in $\CBP$ with $m=2$, and a realization of the corresponding structure in $\CTBP$ that generates it. Notice that different realizations in $\CTBP$ can generate the same graph in $\CBP$.}
		\label{fig:tree-cbp}
\end{figure}

For example, a tree in $\CBP$ with depth $k$ and vertices of minimum degree $m$ is generated by {\em chains} of individuals in the corresponding $\CTBP$. In \cite{CarGarHof}, we prove that the number of such trees in PAM diverges as the size of the graph increases. In terms of the $\CTBP$, it is necessary to look for structures similar to the one in Figure \ref{fig:tree-cbp}.
It would be interesting to investigate the topological properties of the neighborhoods of vertices in $\CBP$, to see if and how they depend on the corresponding $\CTBP$. It would also be interesting to compare the local structure of $\CTBP$ with the result given in \cite{Borgs14} in terms of local weak convergence.
\medskip

\paragraph{\bf Random out-degree.} An interesting extension of the present work is the case of {\em random out-degree} graphs. Instead, our $\CBP$ have {\em fixed out-degree} $m\geq2$. However, the collapsing procedure is well defined for any sequence of out-degrees $(m_n)_{n\in\N}$, both deterministic or random. Results are known for PAM with random out-degree (see \cite{Dei}), suggesting that $\CBP$ with random out-degree is the continuous-time version of PAM with random out-degrees.
\medskip

\paragraph{\bf More general PA functions and fitnesses.} When collapsing, the degree $\Din_n(t)$ of a vertex $n$ in $\CBP$ is distributed as the sum of $m$ independent birth process $(\xi_t)_{t\geq0}$. When we consider an affine PA function of the type $f(k) = ak+b$, with $a\geq0$ and $b>0$, the sum of the $m$ weights corresponding to the $m$ individuals becomes $a(D_1+\ldots+D_m)+ mb$, i.e., the collapsed individuals become {\em indistinguishable}. This is still true when we consider an affine PA function $f$ and aging $g$, because of the linearity of $f$ and the fact that the error given by the difference of birth times is negligible.

 This is no longer true when the PA function is not affine and/or in the presence of fitness. In fitness models, every individual $x$ is assigned an independent realization $Y_x$ from a fitness distribution, and it produces children according to the sequence of PA weights $(Y_xf(k))_{k\in\N}$ (see \cite{Borgs}, \cite{Bianconi},\cite{GarvdHW}). In this case, individuals with different fitness values are not indistinguishable anymore. Assigning the same fitness value to $m$ different individuals would define a process that is not a CTBP in the sense of Definition \ref{def-BP}. To overcome this problem, in the case of discrete-valued fitness, we might collapse individuals according to their fitness values and not according to their birth order. This might be applied also to CTBPs with fitness and aging as introduced in \cite{GarvdHW}. This is a topic for future work.

\section{Overview of the proof of Theorem \ref{theo-general}}
\subsection{General branching process theory}
\label{sec-bptheory}
Here we recall the main results on branching processes that we will use in this paper. Continuous-time branching processes (CTBPs) are models where a population is composed by individuals that produce children according to i.i.d. copies of a birth process $(\xi_t)_{t\geq0}$. The formal definition of a CTBP is the following:

\begin{Definition}[Branching process]
\label{def-BP}
We define the set of individuals in the population as
\eqn{
	\mathcal{N} = \bigcup_{n\in\N}\N^n.
}
Consider a point process $\xi$. Then, the continuous-time branching process is described by
\eqn{
\label{space}
	(\Omega,\mathcal{A},\pr) = \prod_{x\in\mathcal{N}}(\Omega_x,\mathcal{A}_x,\pr_x),
}
where $(\Omega_x,\mathcal{A}_x,\pr_x)$ are probability spaces and $(\xi^x)_{x\in\mathcal{N}}$ are i.i.d. copies of $\xi$. For $x\in\N^n$ and $k\in\N$ we denote the $k$-th child of $x$ by $xk\in\N^{n+1}$. More generally, for $x\in\N^n$ and $y\in\N^m$, we denote the $y$ descendant of $x$ by $xy$. We call the branching process the triplet $(\Omega,\mathcal{A},\pr)$ and the sequence of point processes $(\xi^x)_{x\in\mathcal{N}}$. We denote the branching process by $\sub{\xi}$.
\end{Definition}

The behavior of CTBPs is determined by the properties of the birth process.
Consider a jump process $\xi$ on $\R^+$, i.e., an integer-valued random measure on $\R^+$. Denote the time of the $k$-th jump of $(\xi_t)_{t\geq0}$ by $T_k$. Then we say that $\xi$ is {\em supercritical} when there exists $\alpha^*>0$ such that
\eqn{
\label{for:supercr}
	\mathcal{L}(\E\xi(d\cdot))(\alpha^*) = \int_0^\infty \e^{-\alpha^* t}\E\xi(dt) =1.
}

Here $\E\xi(dx)$ denotes the density of the {\itshape averaged} measure $\E[\xi([0,t])]$. A second fundamental property for the analysis of branching processes is the Malthusian property.  Consider a point process $\xi$. Take the parameter $\alpha^*$ that satisfies\eqref{for:supercr}. Then the process $\xi$ is {\em Malthusian} with Malthusian parameter $\alpha^*$ if 
\eqn{
\label{for-mu}
	\mu := \left.-\frac{d}{d\alpha}\left(\mathcal{L}(\E\xi(d\cdot))\right)(\alpha)\right|_{\alpha^*} = \int_0^\infty t\e^{-\alpha^* t}\E\xi(dt) <\infty. 
}

An important class of functions of branching processes are random characteristics: 
\begin{Definition}[Random characteristic]
\label{def-charact}
	A {\em random characteristic} is a real-valued process $\Phi\colon \Omega\times\R\rightarrow\R$ such that $\Phi(\omega,s)=0$ for any $s<0$, and $\Phi(\omega,s) = \Phi(s)$ is a deterministic bounded function for every $s\geq 0$ that only depends on $\omega$ through the birth time of the individual, as well as the birth process of its children.
\end{Definition}

Let  $\La(f(\cdot))(\alpha)$ denote the Laplace transform of a function $f$ evaluated in $\alpha>0$. We are now ready to quote the main result on CTBPs:

\begin{Theorem}
\label{the-BPmain}
Consider a point process $\xi$, and the corresponding branching process $\sub{\xi}$. Let $\xi$ be supercritical and Malthusian with parameter $\alpha^*$, and suppose that there exists $\bar{\alpha}<\alpha^*$ such that
\eqn{
	\int_0^\infty \e^{-\bar{\alpha}t}\E\xi(dt)<\infty.
}
Then, the following properties hold:
\begin{enumerate}
		\item There exists a random variable $\Theta$ such that, for any random characteristic $\Psi$, as $t\rightarrow\infty$,
		\eqn{
		\label{th-expogrowth-f1}
			\e^{-\alpha^* t}\sub{\xi}^\Psi_t\stackrel{\pr-a.s.}{\longrightarrow}\frac{1}{\mu}\La(\E[\Psi(\cdot)])(\alpha^*)\Theta.
		}	
		\item On the event 
					$\{\sub{\xi}^{\I_{\R^+}}_t\rightarrow\infty\}$, $\pr\left(\Theta>0\right)=1$ and  $\E[\Theta]=1$. 
	\end{enumerate}
\end{Theorem}
These results are classical (see \cite{Athr2}, \cite{RudValko}, \cite{Jagers}, \cite{Nerman}). 
From \eqref{th-expogrowth-f1} it follows immediately that, for any two random characteristics $\Phi$ and $\Psi$, as $t\rightarrow\infty$,
\eqn{
	\label{th-expogrowth-f2}
	\frac{\sub{\xi}^{\Phi}_t}{\sub{\xi}^{\Psi}_t}
		\stackrel{\pr-\mbox{a.s.}}{\longrightarrow}
			\frac{\mathcal{L}(\E[\Phi(\cdot)])(\alpha^*)}{\mathcal{L}(\E[\Psi(\cdot)])(\alpha^*)}.
}
As a consequence, the ratio between the branching process evaluated with the two characteristics $\I_{\{k\}}$ and $\I_{\R^+}$, which is the fraction of individuals with $k$ children, converges to a deterministic limit. We denote this limit by $(p^{\sss(1)}_k)_{k\in\N}$, where
\eqn{
\label{for-pkCTBP}
	p^{\sss(1)}_k = \alpha^*\mathcal{L}(\pr\left(\xi(\cdot)=k\right))(\alpha^*) = \alpha^*\int_0^\infty \e^{-\alpha^* t}\pr\left(\xi(t)=k\right)dt
		= \E\left[\pr(\xi(u)=k)_{u=T_{\alpha^*}}\right].
}
Here $T_{\alpha^*}$ is an exponential random variable with rate $\alpha^*$ independent of $\xi$. Then $(p^{\sss(1)}_k)_{k\in\N}$ is called the {\em limiting degree distribution} for the branching process $\sub{\xi}$. The notation $p^{\sss(1)}_k$ underlines the fact that the CTBP can be seen as a CBP where we fix $m=1$.

\subsection{Structure of the proof of Theorem \ref{theo-general}}
\label{sec-proofstrct}
Our main result requires the following condition:
\begin{Condition}[Lipschitz]
\label{cond-bounded}
	Assume that a birth process $(\xi_t)_{\geq0}$ is supercritical and Malthusian. The {\em  Lipschitz} condition is that, for every $k\in\N$, there exists a constant $0<L(k)<\infty$ such that the function $P_k[\xi](t) = \pr\left(\xi_t = k\right)$ is Lipschitz with constant $L(k)$.
\end{Condition}
Condition \ref{cond-bounded} requires that the functions $(P_k[\xi](t))_{k\in\N}$ associated to the birth process $(\xi_t)_{t\geq0}$ are smooth, in the sense that they do not have dramatic changes over time. We can now state the main result of the paper:
\begin{Theorem}
\label{theo-general}
Let $(\xi_t)_{t\geq0}$ be a supercritical and Malthusian birth process that satisfies Condition \ref{cond-bounded}. Let $(\CBP^{\sss(m)}_t)_{t\geq0}$ be the corresponding collapsed branching process. Let $\Theta$ and $\mu$ be as in Theorem \ref{the-BPmain}. 
 Denote the size of $\CBP^{\sss(m)}_t$ by $N^{\sss(m)}(t)$, and the number of vertices with degree $k$ by $N^{\sss(m)}_k(t)$. Then, as $t\rightarrow\infty$,
\begin{enumerate}
	\item \eqn{
	\label{for:mainT-1}
		m\e^{-\alpha^* t}N^{\sss(m)}(t)\stackrel{\pr-a.s.}{\longrightarrow} \frac{1}{\mu \alpha^*}\Theta;
	}
	\item for every $k\in\N$, there exists $p^{\sss(m)}_k$ such that,
	\eqn{
	\label{for:mainT-2}
		m\e^{-\alpha^* t}N^{\sss(m)}_k(t)\stackrel{\pr}{\longrightarrow} \frac{1}{\mu\alpha^*}p^{\sss(m)}_k\Theta;
	}
	\item As a consequence, 
	\eqn{
	\label{for:mainT-3}
		\frac{N^{\sss(m)}_k(t)}{N^{\sss(m)}(t)}\stackrel{\pr}{\longrightarrow} p_k^{\sss(m)}.
	}
	\end{enumerate}
\end{Theorem}
\noindent 
The sequence $(p_k^{\sss(m)})_{k\in\N}$ is called the {\em limiting degree distribution} of $(\CBP^{\sss(m)}_t)_{t\geq0}$, and is given by
\eqn{
\label{for-general}
	p_k^{\sss(m)} = \alpha^*\mathcal{L}\left(P[\xi](\cdot)^{*m}_k\right)(\alpha^*) = \E\left[P[\xi](T_{\alpha^*})^{*m}_k\right],
} 
where $P_k[\xi](t) = \pr(\xi_t = k)$,
 $T_{\alpha^*}$ is an exponentially distributed random variable with parameter $\alpha^*$, and 
\eqn{
\label{for-convXI}
	P[\xi](t)^{*m}_k = \sum_{k_1+\cdots+k_m=k}P_{k_1}[\xi](t)\cdots P_{k_m}[\xi](t)
}
is the $k$-th element of the $m$-fold convolution of the sequence $(P_k[\xi](t) )_{k\in\N}$.

We now comment on Theorem \ref{theo-general} (for comparison with CTBPs, we refer to Theorem \ref{the-BPmain}). \eqref{for:mainT-1} assures us that the size of a $\CBP$ grows at exponential rate $\alpha^*$ as for the underlying $\CTBP$. Even the size of $\CBP^{\sss(m)}_t$, up to the constant $m$, scales exactly as the size of the $\CTBP$, and the limiting random variable $\Theta$ is the same. This means that the collapsing procedure does not destroy the exponential growth of the graph. 

\eqref{for:mainT-2} assures that, for every $k\in\N$, the number of vertices with in-degree $k$ scales exponentially and also in this case we have a limiting random variable. \eqref{for:mainT-3} tells us that there exists a {\em deterministic} limiting degree distribution for a $\CBP$.

The expression for $(p_k^{\sss(m)})_{k\in\N}$ can be explained in terms of CTBPs. In fact, for a CTBP $\sub{\xi}$, the limiting degree distribution is given by $p^{\sss(1)}_k = \E\left[P_k[\xi](T_{\alpha^*})\right]$,
with $\alpha^*$ the Malthusian parameter of $\sub{\xi}$. We can see $T_{\alpha^*}$ as a {\em time unit} that a process $(\xi_t)_{t\geq0}$ takes to generate, on average, $1$ individual. Then, $p^{\sss(1)}_k$ can be seen as the probability that $(\xi_t)_{t\geq0}$ generates $k$ individuals instead of the average $1$. Using the same heuristic, the limiting degree distribution of $\CBP$ can be seen as the probability that $m$ different individuals produce $k$ children in total in the time unit $T_{\alpha^*}$. Notice that in the expression of $(p_k^{\sss(m)})_{k\in\N}$ the Malthusian parameter $\alpha^*$ is that of the branching process $\sub{\xi}$. 

Unfortunately, the size of CBP and the number of vertices with degree $k\in\N$ are not the evaluation of a CTBP with a random characteristic as in Definition \ref{def-charact}. For example the degree of a vertex in $\CBP$ is the sum of the degrees of $m$ different individuals. The solution for the size of CBP and the number of vertices with degree $k$ is different.
From Definition \ref{def-collBP}, it is obvious that
\eqn{
\label{eq-dimens}
	N^{\sss(m)}(t) = \left\lceil \frac{\sub{\xi}^{\I_{\R^+}}_t}{m}\right\rceil.
}
Using then \eqref{th-expogrowth-f1}, the proof of  \eqref{for:mainT-1} is immediate.

The proof of  \eqref{for:mainT-2} is harder, and it requires a conditional second moment method on $N^{\sss(m)}_k(t)$. Before stating the result, we need a preliminary discussion. We use artificial randomness that we add to the branching process to rewrite the degree of a vertex in CBP in terms of a random characteristic. In the population space in the definition of CTBPs, we consider a single birth process $(\xi^x_t)_{t\geq0}$ for every individual $x$ in the population. We instead consider on every $\Omega_x$ a vector of birth processes
$(\xi^{x,1}_t,\ldots,\xi^{x,m}_t)$,
where $\xi^{x,1}_t,\cdots,\xi^{x,m}_t$ are i.i.d.\ copies of the birth process, defined on the space corresponding to the individual $x$. With this notation, the standard branching processes defined by $(\xi_t)_{t
\geq0}$ is the branching process where we consider as birth process the first component of every vector associated to every individual. 

Now, for $k\in\N$, we consider the random characteristic 
\eqn{
\label{for-PHI}
	\Phi^{\sss(m)}_k(t) = \I_{\{k\}}\left(\xi^{x,1}_{t-\tau_x}+\cdots+\xi^{x,m}_{t-\tau_x}\right),
}
which corresponds to the event that the sum of the components of the vector associated to the individual $x$ when its age is $t-\tau_x$ is equal to $k$. This is a random characteristic that depends only on the randomness defined on the space $\Omega_x$. 
\begin{figure}[t]
	\centering
	\includegraphics[width=0.65\textwidth]{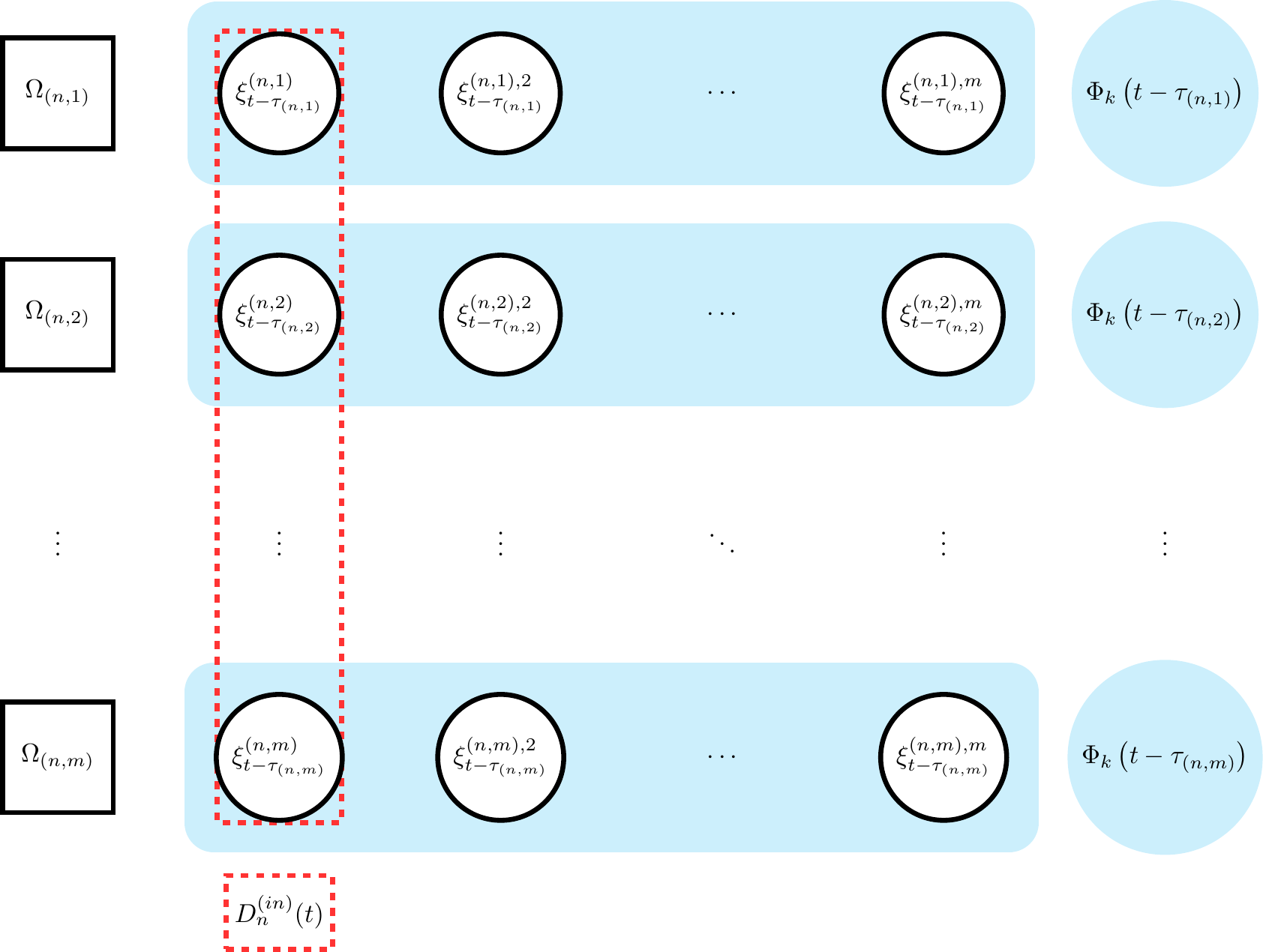}
\end{figure}

The crucial observation is that 
\eqn{
\label{for-crucial}
\begin{split}
	\pr\left(\Din_n(t)=k\right) & = \pr\left(\xi^{(n,1)}_{t-\tau_{(n,1)}}+\cdots+\xi^{(n,m)}_{t-\tau_{(n,m)}}= k\right) \\
	&  \approx \frac{1}{m}\sum_{j=1}^m\pr\left(\xi^{(n,j),1}_{t-\tau_{(n,j)}}+\cdots+\xi^{(n,j),m}_{t-\tau_{(n,j)}}= k\right) \\
	& = \frac{1}{m}\sum_{j=1}^m \E\left[\Phi^{\sss(m)}_k(t-\tau_{(n,j)})\right]+(\mathrm{error}),
\end{split}
}
when we assume that the difference between the birth times $\tau_{(n,1)},\tau_{(n,2)},\ldots,\tau_{(n,m)}$ is very small. The approximation in \eqref{for-crucial} can be explained by the fact that all the components of the vectors $(\xi^{n,1}_t,\ldots,\xi^{n,m}_t)$ are i.i.d. and $\tau_{(n,1)}\approx \tau_{(n,m)}$. In fact, on the left side of \eqref{for-crucial} we have the probability that the sum of $m$ independent copies of $(\xi_t)_{t\geq0}$, evaluated at different times, is equal to $k$. Assuming that the differences between the birth times $\tau_{(n,1)},\tau_{(n,2)},\ldots,\tau_{(n,m)}$ are small, we can just evaluate the $m$ different processes at time $\tau_{(n,1)}$, with a negligible error.

The proof of this, based on Condition \ref{cond-bounded}, is given in Proposition \ref{prop-apprfixt}. It gives the bound on the error term with the difference between the birth rimes of the individuals collapsed to generate the vertex, i.e., the error term is bounded by $mL|\tau_{(n,m)}-\tau_{(n,1)}|$, where $L = \max_{i\in[k]}\{L(i)\}$.

The use of artificial randomness might not seem intuitive. The point is that the equality in expectation between the random characteristic $\Phi^{\sss(m)}_k(t-\tau_{(n,1)})$ and $\Din_n(t)$ is enough. This relies on the fact that, conditionally on the first stages of the branching process, the contribution to the number of vertices with degree $k$ given by the latter individuals is almost deterministic. Let us formalize this idea:

\begin{Definition}[$x$-bulk filtration]
\label{def-bulk}
Consider a branching process $\sub{\xi}$, and its natural filtration $(\mathcal{F}_t)_{t\geq0}$. Consider an increasing function $x(t):\R^+\rightarrow\R^+$.  We call $(\mathcal{F}_{x(t)})_{t\geq0}$ the {\em $x$-bulk filtration} of $\sub{\xi}$. At every time $t\geq0$, a random variable measurable with respect to $\F_{x(t)}$ is called {\em $x$-bulk measurable}.
\end{Definition}

If we consider $x(t)$ to be $o(t)$, then the $x$-bulk filtration heuristically contains information only on the early stage of the CTBP. Nevertheless, the information contained in $\F_{x(t)}$ is enough to estimate the behavior of the CTBP:

\begin{Proposition}[Conditional moments of $N^{\sss(m)}_k(t)$]
\label{prop-condmoments}
Assume that $x$ is a monotonic function such that, as $t\rightarrow\infty$,  $x(t)\rightarrow\infty$ and $x(t) = o(t)$. Then, under the conditions of Theorem \ref{theo-general}, as $t\rightarrow\infty$, 
\begin{enumerate}
	\item \eqn{
	\label{for-aim-2}
			m\e^{-\alpha^* t}\E\left[\left.N^{\sss(m)}_k(t)\right|\F_{x(t)}\right] 
				\stackrel{\pr-a.s.}{\longrightarrow} \frac{1}{\mu}\La\left(\Phi^{\sss(m)}_k(\cdot)\right)(\alpha^*)\Theta;
		}
	\item 
		\eqn{
		\label{for-aim-3}
			\e^{-2\alpha^* t}\E\left[\left.N^{\sss(m)}_k(t)^2\right|\F_{x(t)}\right] 
				\stackrel{\pr-a.s.}{\leq} 
				\left(\e^{-\alpha^* t}\E\left[\left.N^{\sss(m)}_k(t)\right|\F_{x(t)}\right]\right)^2 + o(1).
		}
\end{enumerate}
\end{Proposition}
We point out that if $X\leq Y+o(1)$, then $o(1)$ is a term that converges almost surely to 0. The proof of Proposition \ref{prop-condmoments} is moved to Section \ref{sec-momentMethod}. With Proposition \ref{prop-condmoments} in hand, we can prove \eqref{for:mainT-2}. We bound $\left| m\e^{-\alpha^* t}N_k^{\sss(m)}(t)-\frac{1}{\mu}\La(\Phi^{\sss(m)}_k(\cdot))(\alpha^*)\Theta\right| $ by
\eqn{
\label{for-mainT-concl-1}
	\left| m\e^{-\alpha^* t}N_k^{\sss(m)}(t)- m\e^{-\alpha^* t}\E[N^{\sss(m)}_k(t)|\F_{x(t)}]\right|+
		\left| m\e^{-\alpha^* t}\E[N^{\sss(m)}_k(t)|\F_{x(t)}]-\frac{1}{\mu}\La(\Phi^{\sss(m)}_k(\cdot))(\alpha^*)\Theta\right|.
}
As a consequence, \eqref{for:mainT-2} holds if both terms in \eqref{for-mainT-concl-1} converges $\pr$-a.s. to zero. For the second term this is true by \eqref{for-aim-2}. For the first term, we use \eqref{for-aim-2} and \eqref{for-aim-3} to conclude that $\mathrm{Var}\left(m\e^{-\alpha^* t}N^{\sss(m)}_k(t)|\F_{x(t)}\right) = o_{a.s.}(1)$, so that
\eqn{
\label{for-condVariance}
	\left| m\e^{-\alpha^* t}N_k^{\sss(m)}(t)- m\e^{-\alpha^* t}\E[N^{\sss(m)}_k(t)|\F_{x(t)}]\right| \stackrel{\pr}{\displaystyle \longrightarrow}0.
}
This concludes the proof of \eqref{for:mainT-2}. \eqref{for:mainT-3} follows immediately.

\begin{Remark}[Times and bulk sigma-field]
\label{remark:bulk:function}
{\rm We have proved Proposition \ref{prop-condmoments} (and thus Theorem \ref{theo-general}) by looking at the CTBP at time $t$, considering the $x(t)$-bulk sigma-field. We can extend the argument as follows. Consider $s\geq0$, and let $y\colon \R^+\rightarrow\R^+$ be a  monotonic function of $s$ such that $y(s)/s\rightarrow\infty$ as $s\rightarrow \infty$. In this case, looking at the graph at time $y(s)$, and considering the $s$-bulk sigma-field, Proposition \ref{prop-condmoments} still holds. More generally, as suggested by \eqref{for-heur-1} below, conditionally on the $s$-bulk sigma-field, the evolution of a CTBP is almost deterministic. This implies that Proposition \ref{prop-condmoments} even holds when we consider a {\em random} process $Y(s)$ such that $Y(s)/s\stackrel{a.s.}{\rightarrow}\infty$, under the assumption that $Y(s)$ is $s$-bulk measurable for every $s\geq0$. These observations will be useful when extending our results to discrete time in Section \ref{sec:discrete}.}
\end{Remark}

\section{Preliminaries on birth times}
\label{sec-times}
\subsection{Bound on the difference in time}
\label{sec-difftime}
In this section, we prove the fact that the error term in \eqref{for-crucial} can be bounded by the difference of birth times of the considered individuals. We introduce the definition of convolution:
\begin{Definition}[Convolution]
\label{def-conv}
We define convolution between two sequences $(a_k)_{k\in\N}$ and $(b_k)_{k\in\N}$ as
\eqn{
	(a*b)_k:=  \sum_{l=0}^k a_lb_{k-l}.
}
\end{Definition}
With this definition, we can state a technical lemma we need to prove the bound we are interested in:
\begin{Lemma}[Difference in times]
\label{lem-differencetime}
Consider the sequence of functions $(P[\xi]_k(t))_{k\in\N}$. If $(\xi_t)_{t\geq0}$ satisfies Condition \ref{cond-bounded}, then, for every $x\in\R^+$, and for every $h_i\leq x$ for $i\in[n]$, 
\eqn{
\label{for-differencetime}
	\left|\left(P[\xi](x-h_1)*\cdots*P[\xi](x-h_m)\right)_k -\left(P[\xi](x-h_1)^{*m}\right)_k\right| \leq L\sum_{j=2}^{m}|h_1-h_j|,
}
where $L = \max_{i\in[k]}L(i)$.
\end{Lemma}
\begin{proof}
Without loss of generality, assume $0\leq h_1\leq \ldots\leq h_n$. We prove Lemma \ref{lem-differencetime} by induction on $m$. We start the induction with $m=2$, so
\eqn{
\label{for-cnvl-1}
	\left(P[\xi](x-h_1)*P[\xi](t-h_2)\right)_k = \sum_{l=0}^kP[\xi]_l(x-h_1)P[\xi]_{k-l}(x-h_2).
}
We now use Condition \ref{cond-bounded} to bound $\left|P[\xi]_{k-l}(x-h_2) - P[\xi]_{k-l}(x-h_1)\right|\leq +L(k-l)(h_2-h_1)$.
Using this in
\eqref{for-cnvl-1}, then  we obtain, for $L= \max_{i\in[k]}L(i)$, 
\eqn{
\label{for-cnvl-5}
	\left|\left(P[\xi](x-h_1)*P[\xi](t-h_2)\right)_k-\left(P[\xi](x-h_1)^{*2}\right)_k\right|  
\leq L\sum_{l=0}^kP[\xi]_{k-l}(x-h_1)\left|(h_2-h_1)\right|.  
}
Since $\sum_{l=0}^kP_l[\xi](x-h_1)= P[\xi]_{\leq k}(x-h_1)\leq 1$, 
$$
	\left|\left(P[\xi](x-h_1)*P[\xi](t-h_2)\right)_k-\left(P[\xi](x-h_1)^{*2}\right)_k\right|\leq L|h_2-h_1|,
$$
so \eqref{for-differencetime} holds for $m=2$. We now advance the induction hypothesis, so suppose that \eqref{for-differencetime} holds for $m-1$. We can write
\eqn{
\label{for-cnvl-6}
	\left(P[\xi](x-h_1)*\cdots*P[\xi](x-h_m)\right)_k  = \sum_{l=0}^k\left(P[\xi](x-h_1)*\cdots*P[\xi](x-h_{m-1})\right)_lP[\xi]_{k-l}(x-h_m).
}
Notice that we can apply \eqref{for-differencetime} to the first terms in the sum in \eqref{for-cnvl-6} thanks to the induction hypothesis, since it is now the convolution of $m-1$ functions. We just need to replace $P[\xi]_{k-l}(x-h_m)$ by $P[\xi]_{k-l}(x-h_1)$. It is easy to do this using a similar argument used to prove the bound in \eqref{for-cnvl-5}, which implies again the use of Condition \ref{cond-bounded}. In the end, we have 
\eqn{
\left|\left(P[\xi](x-h_1)*\cdots*P[\xi](x-h_m)\right)_k -\left(P[\xi](x-h_1)^{*m}\right)_k\right| \leq L\sum_{j=2}^{m-1}|h_1-h_j|+L|h_m-h_1|,
}
where the $m-1$ terms comes from the induction hypothesis, and the last one from the approximation of $P[\xi]_{k-l}(x-h_m)$. This completes the proof.
\end{proof}

Lemma \ref{lem-differencetime} holds for every time $t$ and $h_1,\ldots,h_m$ that we consider.  We can now prove the bound on the error term in \eqref{for-crucial}:

\begin{Proposition}[Approximation at fixed time]
\label{prop-apprfixt}
	Consider $(\CBP^{\sss(m)}_t)_{t\geq0}$ obtained from a branching process $\sub{\xi}$. Assume that $(\xi_t)_{t\geq0}$ satisfies Condition \ref{cond-bounded}. Then, for every $k\in\N$, with $L$ as in Lemma \ref{lem-differencetime}, $\pr$-a.s. for every $n\in\N$,
	\eqn{
	\label{for-apprfixt}
		\left|\pr\left(\Din_n(t)=k~|~\tau_{(n,1)},\ldots,\tau_{(n,m)}\right)-\left(P[\xi](t-\tau_{(n,1)})^{*m}\right)_k\right| \leq Lm|\tau_{(n,m)}-\tau_{(n,1)}|.
	}
\end{Proposition}
\begin{proof}
Conditionally on the birth times, the processes $(\xi^{(n,1)}_t)_{t\geq0},\ldots,(\xi^{(n,m)}_t)_{t\geq0}$ are independent. As a consequence,
\eqn{
\label{for-cond-tau-conv}
\pr\left(\Din_n(t)=k~|~\tau_{(n,1)},\ldots,\tau_{(n,m)}\right) = \left(P[\xi](t-\tau_{(n,1)})*\cdots*P[\xi](t-\tau_{(n,m)})\right)_k.
}
Then \eqref{for-apprfixt} follows immediately from Lemma \ref{lem-differencetime}, where we consider $h_1 = \tau_{(n,1)},\ldots,h_m = \tau_{(n,m)}$, and the fact that $\tau_{(n,j)}-\tau_{(n,1)}\leq \tau_{(n,m)}-\tau_{(n,1)}$ for every $j=1,\ldots,m$.
\end{proof}

\subsection{Replacing birth times with $\mathcal{F}_t$-measurable approximations}
\label{sec-remark-btim}
Recall that $\F_t$ denotes the natural filtration of the CTBP up to time $t$. It is possible to rewrite \eqref{th-expogrowth-f1} as
$$
	n\e^{-\alpha^* \tau_n}\stackrel{\pr-a.s.}{\longrightarrow}\frac{1}{\mu\alpha^*}\Theta.
$$
As a consequence, as $n\rightarrow\infty$, 
\eqn{
\label{for-taun-conv}
	-\tau_n +\frac{1}{\alpha^*}\log n\stackrel{\pr-a.s.}{\longrightarrow}\frac{1}{\alpha^*}	
			\log\left(\frac{1}{\mu \alpha^*}\Theta\right).
}
Notice that on the event $\{\sub{\xi}_t^{\I_{\R^+}}\rightarrow\infty\}$, $\Theta$ is positive with probability $1$, so $\log\left(\frac{1}{\mu \alpha^*}\Theta\right)$ is well defined. Define, for $n\geq \sub{\xi}^{\I_{\R^+}}_t$,
\eqn{
\label{for-sigmaDEF}
	\sigma_n(t) := \frac{1}{\alpha^*}\log n -\frac{1}{\alpha^*}\log\left(\frac{1}{\mu\alpha^*}\Theta_t\right),\quad
			\mbox{ where } \quad\Theta_t = \mu\alpha^*\e^{-\alpha^* t}\sub{\xi}^{\I_{\R^+}}_t.
}
Then $\sigma_n(t)$ is an approximation of $\tau_n$ given the information up to time $t$, where the factor $\Theta_t$ includes the stochastic fluctuation of the size of the branching process.
 What is interesting is that the random variable $\sigma_n(t)$ is an approximation of $\tau_n$ {\em measurable with respect to} $\F_t$. We now prove that $(\sigma_n(t))_{t\geq0}$ is an acceptable approximation of $\tau_n$:

\begin{Lemma}[Error of $(\sigma_n(t))_{t\geq0}$]
\label{lem-sigma}
$\pr$-a.s., as $t\rightarrow\infty$,
\eqn{
\label{for:times-1}
		\sup_{n\geq \sub{\xi}^{\I_{\R^+}}_t}\left|\sigma_n(t)-\tau_n\right|\rightarrow0.
}
\end{Lemma}
\begin{proof}
For every $t\geq0$ and $n\geq \sub{\xi}^{\I_{\R^+}}_t$ we write 
\eqn{
\label{for:times-2}
	\left|\sigma_n(t)-\tau_n\right|\leq \left|\frac{1}{\alpha^*}\log n-\tau_n-\log \left(\frac{1}{\mu\alpha^*}\Theta\right)\right|+
		\left|\log\left(\frac{1}{\mu\alpha^*}\Theta\right)-\log \left(\frac{1}{\mu \alpha^*}\Theta_t\right)\right|.
}
Using \eqref{for:times-2} in \eqref{for:times-1}, we can bound
\eqn{
\label{for:times-3} 
	\sup_{n\geq \sub{\xi}^{\I_{\R^+}}_t}\left|\sigma_n(t)-\tau_n\right|\leq \left|\log\left(\frac{1}{\mu\alpha^*}\Theta\right)-\log \left(\frac{1}{\mu \alpha^*}\Theta_t\right)\right|+
		\sup_{n\geq \sub{\xi}^{\I_{\R^+}}_t} \left|\frac{1}{\alpha^*}\log n-\tau_n-\log \left(\frac{1}{\mu\alpha^*}\Theta\right)\right|.
}
First of all, from \eqref{th-expogrowth-f1} we know $\Theta_t/(\mu\alpha^*)= \e^{-\alpha^* t}\sub{\xi}^{\I_{\R^+}}_t\rightarrow \Theta/(\mu\alpha^*)$. As a consequence, the first term in the right hand side of \eqref{for:times-3} converges $\pr$-a.s. to zero. For the second term, we use \eqref{for-taun-conv} and the fact that the supremum decreases as $\sub{\xi}^{\I_{\R^+}}_t\rightarrow\infty$. This completes the proof.
\end{proof}

Lemma \ref{lem-sigma} suggests that, conditionally on $\F_t$, we can replace the birth sequence $(\tau_n)_{n\geq \sub{\xi}_t^{\I}}$ with the sequence $(\sigma_n(t))_{n\geq \sub{\xi}_t^{\I}}$ when evaluating random characteristics. 

\section{Second moment method: proof of Proposition \ref{prop-condmoments}}
\label{sec-momentMethod}
\subsection{First conditional moment asymptotics}
\label{sec-cond-nkfirst}
In this section, we investigate the first conditional moment of $N^{\sss(m)}_k(t)$ with respect to the bulk filtration. In particular, consider a function $x$ such that, as $t\rightarrow\infty$, $x(t)\rightarrow\infty$ and $x(t)= o(t)$. Heuristically, we want to show that 
\eqn{
\label{for-heur-1}
	m\E\left[N^{\sss(m)}_k(t) ~|~ \F_{x(t)}\right] \approx N^{\sss(m)}(x(t))\E\left[\sub{\xi}^{\Phi^{\sss(m)}_k}_{t-x(t)}\right].
}
Equation \eqref{for-heur-1} shows that, conditionally on the information up to time $x(t)$, at time $t$ we have $N^{\sss(m)}(x(t))$ processes, each one producing the expected number of vertices with degree $k$ at time $t-x(t)$. This follows from the fact that all the individual processes in $\sub{\xi}$ are independent from each other once we condition on the birth times.

We start writing $N^{\sss(m)}_k(t)$ as sum of indicator functions, i.e.,
$$
	\E\left[N^{\sss(m)}_k(t) ~|~ \F_{x(t)}\right] = \E\left[\left.\sum_{n=1}^{N^{\sss(m)}(x(t))}\I_{\{\Din_n(t)=k\}}+\sum_{n=N^{\sss(m)}(x(t))+1}^\infty \I_{\{\Din_n(t)=k\}}\right|\F_{x(t)}\right].
$$
We can ignore the first sum in the conditional expectation, since
\eqn{
\label{for-boundfirstInd}
	0\leq \e^{-\alpha^* t}\E\left[\left.\sum_{n=1}^{N^{\sss(m)}(x(t))}\I_{\{\Din_n(t)=k\}}\right|\F_{x(t)}\right]\leq \e^{-\alpha^* t}N^{\sss(m)}(x(t)),
}
and, using Theorem \ref{the-BPmain} and the fact that $x(t) = o(t)$,
\eqn{
\label{for-boundfirstInd2}
	\e^{-\alpha^*(t-x(t))}\e^{-\alpha^* x(t)}N^{\sss(m)}(x(t))\stackrel{\pr-a.s.}{\longrightarrow}0.
}
Consider the sequence $(\sigma_n(x(t)))_{n\in\N}^{t\geq0}$ as defined in Section \ref{sec-remark-btim}. This is a sequence of random variables that approximates $(\tau_n)_{n\in\N}$ and it is measurable with respect to the bulk filtration. This means that we can write, for any $n\geq N^{\sss(m)}(x(t))$, 
$$
	\Din_n(t) = \xi^{(n,1)}(t-\sigma_{(n,1)}(x(t)))+\cdots+\xi^{(n,m)}(t-\sigma_{(n,m)}(x(t))).
$$
Now, conditionally on the birth times $\sigma_{(n,1)}(x(t)),\ldots,\sigma_{(n,m)}(x(t))$, the $m$ processes related to the $n$-th vertex  $(\xi^{(n,1)}_t)_{t\geq0},\ldots,(\xi^{(n,m)}_t)_{t\geq0}$ are independent, so the probability that the sum is equal to $k$ is
\eqn{
	\left(P[\xi](t-\sigma_{(n,1)}(x(t)))*\cdots*P[\xi](t-\sigma_{(n,m)}(x(t)))\right)_k,
}
which is a $x$-bulk measurable random variable. As a consequence, 
\eqn{
\label{for-sumConv1mom}
\begin{split}
	& \E\left[\left.\sum_{n=N^{\sss(m)}(x(t))+1}^\infty \I_{\{\Din_n(t)=k\}}\right|\F_{x(t)}\right] \\
	& = \sum_{n= N^{\sss(m)}(x(t))+1}^\infty
		\left(P[\xi](t-\sigma_{(n,1)}(x(t)))*\cdots*P[\xi](t-\sigma_{(n,m)}(x(t)))\right)_k.
\end{split}
}
For any $k\in\N$, the function $u\mapsto P_k[\xi](u)$ is zero for negative argument. As a consequence, the sum in \eqref{for-sumConv1mom} is taken only over indeces $n$ such that $\sigma_{(n,j)}(x(t))<t$. From the definition of $\sigma_{(n,j)}(x(t))$ as in \eqref{for-sigmaDEF} and the fact that $(n,j) = m(n-1)+j$, it follows that $\sigma_{(n,j)}(x(t))<t$ if and only if 
\eqn{
\label{for-indecesCOnv1mom}
	n<1+j/m+\e^{\alpha^*(t-x(t))}\sub{\xi}^{\I_{\R^+}}_{x(t)}/m = \e^{\alpha^*(t-x(t))}N^{\sss(m)}(x(t))(1+o_{a.s.}(1)),
}
where $o_{a.s.}(1)$ denotes a term that converges $\pr$-a.s. to zero.
Using \eqref{for-indecesCOnv1mom} and then applying Proposition \ref{prop-apprfixt}, for $L$ as in Lemma \ref{lem-differencetime}, we obtain
\eqn{
\label{for-sumtodivide}
 \sum_{n = N^{\sss(m)}(x(t))+1}^{N^{\sss(m)}(x(t))\e^{\alpha^*(t-x(t))}}P[\xi](t-\sigma_{(n,1)}(x(t)))^{*m}_k + Lm\sum_{n = N^{\sss(m)}(x(t))+1}^{N^{\sss(m)}(x(t))\e^{\alpha^*(t-x(t))}}\sigma_{(n,m)}(x(t))-\sigma_{(n,1)}(x(t)),
}
where the difference between \eqref{for-sumConv1mom} and the first sum in \eqref{for-sumtodivide} is bounded in absolute value by the second sum in \eqref{for-sumtodivide}.

Consider the difference $t-\sigma_{(n,1)}(x(t))$. Using the definition of the sequence $(\sigma_n(x(t)))_{n\in\N}$, and recalling that $mN^{\sss(m)}(x(t)) =\sub{\xi}^{\I_{\R^+}}_{x(t)}(1+o_{a.s.}(1))$ (see \eqref{eq-dimens}), it follows that $t-\sigma_{(N^{\sss(m)}(x(t)),1)}(x(t)) = (t-x(t))(1+o_{a.s.}(1))$. As a consequence, ignoring negligible terms, 
\eqn{
\label{for-rewriteTimes}
\begin{split}
	t-\sigma_{(n,1)}(x(t)) & = t-\sigma_{(N^{\sss(m)}(x(t)),1)}(x(t))- \left(\sigma_{(n,1)}(x(t))-\sigma_{(N^{\sss(m)}(x(t)),1)}(x(t))\right) \\
	& =t-x(t)+\frac{1}{\alpha^*}\log\left(\frac{m(n-1)+1}{mN^{\sss(m)}(x(t))}\right)\\
&  = t-x(t)+\frac{1}{\alpha^*}\log\left(\frac{n}{N^{\sss(m)}(x(t))}\right).
\end{split}
}
The second sum in the right hand side of \eqref{for-sumtodivide} is bounded by a telescopic sum,  since $\sigma_{(n,1)}(x(t))\geq \sigma_{(n-1,m)}(x(t))$, which implies that we can bound it with the difference between the last and the first term. 
Using \eqref{for-rewriteTimes} in \eqref{for-sumtodivide}, for $s=t-x(t)$,  it leads to
\eqn{
\label{for-doublesum}
\begin{split}
	&\sum_{n = N^{\sss(m)}(x(t))+1}^{N^{\sss(m)}(x(t))\e^{\alpha^*s}}P[\xi](s-\frac{1}{\alpha^*}\log\left(\frac{m(n-1)+1}{mN^{\sss(m)}(x(t))}\right))^{*m}_k +\frac{mL}{\alpha^*}\log\left(\frac{mN^{\sss(m)}(x(t))\e^{\alpha^* s}}{mN^{\sss(m)}(x(t))}\right) \\
	& =  \sum_{p=1}^{\e^{\alpha^* s}}\sum_{q=1}^{N^{\sss(m)}(x(t))}P[\xi](s-\frac{1}{\alpha^*}\log\left(
		p+q/N^{\sss(m)}(x(t))\right))^{*m}_k + mL(t-x(t))\\
		& = N^{\sss(m)}(x(t))\sum_{p=1}^{\e^{\alpha^* s}}P[\xi]\left(s-\frac{1}{\alpha^*}\log(p)\right)^{*m}_k + mL(t-x(t))\\
		& = N^{\sss(m)}(x(t))\sum_{p=1}^{\e^{\alpha^* s}}\E\left[\Phi^{\sss(m)}_k\left(s-\frac{1}{\alpha^*}\log(p)\right)\right] + mL(t-x(t)).
\end{split}
}
We can ignore the term $ mL(t-x(t))$, since $\e^{-\alpha^* t} mL(t-x(t)) = o(1)$. To analyze the remaining sum, we introduce two measures
 $\gamma_1$ and $\gamma_2$ on $\R^+$. For $v\geq 0$, 
$$
	\gamma_1([0,v]) = \int_0^v\sum_{p\in\N}\delta_{\{1/\alpha^* \log p\}}(du)= \e^{\alpha^*v},~~\mbox{and}~~
	\gamma_2([0,v]) = \E\left[\int_0^v \sum_{n\in\N}\delta_{\{\tau_n\}}(du)\right] = \E\left[\sub{\xi}_v^{\I_{\R^+}}\right].
$$
Notice that $\gamma_2$ is the average measure of the random measure given by the branching process size. From Theorem \ref{the-BPmain} we know that $\gamma_2([0,v]) = \E[\sub{\xi}_v^{\I_{\R^+}}] = (1/\mu\alpha^*)\e^{\alpha^* v}(1+o(1))$. This means that, asymptotically in $v$, $\gamma_1([0,v]) = \mu\alpha^* \gamma_2([0,v])$. Using these two measures it is possible to write
\eqn{
\label{for-equalE}
\begin{split}
	\sum_{p=1}^{\e^{\alpha^* s}}\E\left[\Phi^{\sss(m)}_k\left(s-\frac{1}{\alpha^*}\log(p)\right)\right] & = \int_0^s \E[\Phi^{\sss(m)}_k(s-u)]\gamma_1(du) \\
	& = \mu\alpha^* \int_0^s \E[\Phi^{\sss(m)}_k(s-u)]\gamma_2(du) = \mu\alpha^*\E\left[\sub{\xi}^{\Phi^{\sss(m)}_k}_s\right].  
	\end{split}
}
Using \eqref{for-equalE} in \eqref{for-doublesum}, we conclude that
\eqn{
\begin{split}
	\e^{-\alpha^* t}\E\left[N^{\sss(m)}_k(t) ~|~ \F_{x(t)}\right] & = \e^{-\alpha^* t}\mu\alpha^*N^{\sss(m)}(x(t))\E\left[\sub{\xi}^{\Phi^{\sss(m)}_k}_{t-x(t)}\right] +o_{a.s.}(1) \\
	& = \left(\mu\alpha^*\e^{-\alpha^* x(t)}N^{\sss(m)}(x(t))\right)\left(\e^{-\alpha^* (t-x(t))}\E\left[\sub{\xi}^{\Phi^{\sss(m)}_k}_{t-x(t)}\right]\right) +o_{a.s.}(1).
\end{split}
}
Applying \eqref{th-expogrowth-f1} it follows that, as $t\rightarrow\infty$, $\mu\alpha^*\e^{-\alpha^* x(t)}N(x(t))$ converges $\pr$-a.s. to $\Theta$, while $\mu\alpha^*\e^{-\alpha^* (t-x(t))}\E\left[\sub{\xi}^{\Phi^{\sss(m)}_k}_{t-x(t)}\right]$ converges to $\La(\Phi^{\sss(m)}_k(\cdot))(\alpha^*)/\mu$. This completes the proof of \eqref{for-aim-2}.

\subsection{Conditional second moment asymptotics}
\label{sec-cond-nksecond}
In this section, we prove \eqref{for-aim-3}, i.e., the result on the conditional second moment of $N^{\sss(m)}_k(t)$. 
We again write $N^{\sss(m)}_k(t)$ as sum of indicator functions, which means
\eqn{
	\e^{-2\alpha^* t}\E\left[\left.N^{\sss(m)}_k(t)^2\right|\F_{x(t)}\right] = \e^{-2\alpha^* t}\E\left[\left.
		\sum_{n,n'\in\N}\I_{\{\Din_n(t)=k\}}\I_{\{\Din_{n'}(t)=k\}}\right|\F_{x(t)}\right].
}
We now divide the sum in different sums, according to the indices $n$ and $n'$, as 
\eqn{
\label{for-splitsum}
\begin{split}
	&\sum_{n,n'\leq N^{\sss(m)}(x(t))}\I_{\{\Din_n(t)=k\}}\I_{\{\Din_{n'}(t)=k\}}\\
	& +\sum_{n,n'>N^{\sss(m)}(x(t))}\I_{\{\Din_n(t)=k\}}\I_{\{\Din_{n'}(t)=k\}}+2\sum_{n\leq N^{\sss(m)}(x(t)),n'>N^{\sss(m)}(x(t))}\I_{\{\Din_n(t)=k\}}\I_{\{\Din_{n'}(t)=k\}}.
\end{split}
}
For the first sum in \eqref{for-splitsum}, we use \eqref{for-boundfirstInd} as bound, and by \eqref{for-boundfirstInd2} it is $o_{a.s.}(1)$.
For the second sum in \eqref{for-splitsum}, we again use the sequence $(\sigma_n(x(t)))_{n\in\N}$ to approximate the birth times. Using similar arguments as in Section \ref{sec-cond-nkfirst}, and the fact that conditionally on the birth times all the birth processes are independent, we write, for $n\neq n'$ and $n,n'>N^{\sss(m)}(x(t))$,
\eqn{
\label{for2mom-splitProb}
\begin{split}
\pr\left(\Din_n(t)=k,\Din_{n'}(t)=k~|~\F_{x(t)}\right) & = \left(P[\xi](t-\sigma_{(n,1)}(x(t)))*\cdots*P[\xi](t-\sigma_{(n,m)}(x(t)))\right)_k\\
& \quad \times\left(P[\xi](t-\sigma_{(n',1)}(x(t)))*\cdots*P[\xi](t-\sigma_{(n',m)}(x(t)))\right)_k.
\end{split}
}
We can use \eqref{for2mom-splitProb} to bound the conditional expectation of the second sum in  \eqref{for-splitsum}. In fact, adding the missing terms we can write
\eqn{
\begin{split}
	&\E\bigg[\sum_{n,n'>N^{\sss(m)}(x(t))}\I_{\{\Din_n(t)=k\}}\I_{\{\Din_{n'}(t)=k\}}
				\bigg|\F_{x(t)}\bigg]\\
	& \leq \bigg( \sum_{n>N^{\sss(m)}(x(t))}\left(P[\xi](t-\sigma_{(n,1)}(x(t)))*\cdots*P[\xi](t-\sigma_{(n,m)}(x(t)))\right)_k\bigg)^2 + \E\left[N^{\sss(m)}_k(t)\bigg|\F_{x(t)}\right]\\\
& = 	\E\bigg[\sum_{n>N^{\sss(m)}(x(t))}\I_{\{\Din_n(t)=k\}}\bigg|\F_{x(t)}\bigg]^2+\E\left[\left.N^{\sss(m)}_k(t)\right|\F_{x(t)}\right] \\
& \leq \E\left[\left.N^{\sss(m)}_k(t)\right|\F_{x(t)}\right]^2 +\E\left[\left.N^{\sss(m)}_k(t)\right|\F_{x(t)}\right].
\end{split}
}
The third sum in \eqref{for-splitsum} can be easily bound by $2N^{\sss(m)}(x(t))\E[N^{\sss(m)}_k(t)|\F_{x(t)}]$. Putting together the three bounds we obtained, we have that $\e^{-2\alpha^* t}\E\left[N^{\sss(m)}_k(t)^2|\F_{x(t)}\right]$ is bounded by
\eqn{
\label{for-last2mom}
	\e^{-2\alpha^* t}\E[N_k(t)|\F_{x(t)}]^2 + 
		\e^{-2\alpha^* t}\left(2N(x(t))+1\right)\E[N_k(t)|\F_{x(t)}]+ o_{a.s.}(1).
}
The result follows since the second term in  \eqref{for-last2mom} is again $o_{a.s.}(1)$, similarly to the first term in \eqref{for-splitsum}. 

\section{Proofs of corollaries \ref{cor-PAM}, \ref{cor-RRG} and \ref{cor-aging}}
\subsection{Corollaries \ref{cor-PAM} and \ref{cor-RRG}}
In Section \ref{sec-results} we already showed that CBPs defined by birth processes as in Definition \ref{def-emb_birthpr} embeds the PAM in continuous-time  and what we called random recursive graph. We just need to show that Condition \ref{cond-bounded} is satisfied. In general, processes defined as in Definition \ref{def-emb_birthpr} are differentiable and satisfy a recursive property (see \cite[Section 3.2]{athrBook}):
\eqn{
\label{for-derivative-PAM}
\frac{d}{dt}P_0[\xi](t) = -\lambda_0P_0[\xi](t),\quad \mbox{ and, for }k\geq1,\quad 
	\frac{d}{dt}P_k[\xi](t) = -\lambda_k P_k[\xi](t)+\lambda_{k-1}P_{k-1}[\xi](t).
}
Since in general we consider a non-decreasing sequence $(\lambda_k)_{k\in\N}$, it is possible to see that if we set $L(k) = \lambda_k$ then Condition \ref{cond-bounded} is satisfied. Hence, the limiting degree distribution $(p_k^{\sss(m)})_{k\in\N}$ is the distribution of the sum of $m$ independent copies of $(\xi_t)_{t\geq0}$ at exponential time $T_{\alpha^*}$, for $\alpha^*$ Malthusian parameter of the CTBP. 

In the case of the PAM embedding, the sum of $m$ birth processes is distributed as an embedding birth process defined by the PA rule $\bar{\lambda}_k = k+m+\delta$ (it is easy to prove this by induction over the distribution of birth times).
This implies that we can use known results on this type of birth processes (\cite{RudValko},\cite{Athr}) to write
$$
	p^{\sss(m)}_k = \pr\left(\xi^1_{T_{\alpha^*}}+\cdots+\xi^m_{T_{\alpha^*}}=k\right) = \frac{\alpha^*}{\alpha^* + k+m+\delta}\prod_{i=0}^{k-1}\frac{i+m+\delta}{\alpha^*+i+m+\delta},
$$ 
that can be rewritten as in \eqref{for-PAMdeg} using $\Gamma$ functions, since in this case $\alpha^* = 1+\delta/m$ (see \cite[Section 4.2]{RudValko},  \cite[Proposition 3.15]{GarvdHW}). 

For the random recursive graph, calculations are easier. It is easy to show that in this case $\alpha^*=1$. Since the sum of $m$ Poisson processes (PP) with parameter 1 is a PP with parameter $m$, the limiting degree distribution is the distribution of a 
PP at an exponentially distributed time with parameter 1. Then
\eqn{
	p^{\sss(m)}_k = \E\left[\e^{-mT_1}\frac{(mT_1)^k}{k!}\right] = \frac{1}{m+1}\left(1+\frac{1}{m}\right)^{-k}.
}
As mentioned, for $m=1$ (so without collapsing) the random recursive graph reduces to the random recursive tree, and the limiting distribution is just $p^{\sss(1)}_k = 2^{-(k+1)}$ (see \cite{Janson}).

\subsection{The aging case}
Here we prove the result on aging processes stated in Corollary \ref{cor-aging}. The result follows immediately from the proof of Corollary \ref{cor-PAM} and the definition of the aging process. In fact, an aging process is defined as $(\xi_{G(t)})_{t\geq0}$, where $(\xi_t)_{t\geq0}$ is an embedding process defined by the sequence $(k+1+\delta/m)_{k\in\N}$. As a simple consequence of the chain rule, from \eqref{for-derivative-PAM} it follows that
\eqn{
	\frac{d}{dt}P_k[\xi](G(t)) = \left( -(k+1+\delta/m) P_k[\xi](t)+(k+\delta/m)P_{k-1}[\xi](t)\right)g(t).
}
Assuming that the aging function $g$ is bounded almost everywhere, Condition \ref{cond-bounded} is satisfied for $L = k\sup_{t\geq0}|g(t)|$. The condition $\lim_{t\rightarrow\infty}\E[\xi_{G(t)}]>1$ is necessary and sufficient for the existence of the Malthusian parameter $\alpha^*$ (see \cite[Lemma 4.1]{GarvdHW}). 

Since the sum of $m$ processes $\xi_t^1+\cdots+\xi_t^m$ is distributed as a single embedding process  defined by the sequence $(k+m+\delta)_{k\in\N}$, it follows that $\xi_{G(t)}^1+\cdots+\xi_{G(t)}^m$ is distributed as a single aging process with the same aging function $g$ and  sequence $(k+m+\delta)_{k\in\N}$. \eqref{for-distrAging-CBP} is then a consequence of \cite[Proposition 5.2]{GarvdHW}.

\section{Discrete-time processes: proof of Theorem \ref{th:PAMdiscrete}}
\label{sec:discrete}
The convergence result given in Theorem \ref{the-limitCBP} assures that in continuous time, the proportion of vertices in CBP with degree $k$ converges in probability to $p^{\sss(m)}_k$. When considering a CTBP in the presence of aging, this result is enough since these types of CBPs are defined only in continuous time.

When we instead consider embedding processes as in Definition \ref{def-emb_birthpr}, we can consider a discrete-time sequence of random  graphs $(\CBP^{\sss(m)}_{\tau_n})_{n\in\N}$, where $({\tau_n})_{n\in\N}$ is the sequence of birth times of the corresponding CTBP. This is the way the PAM is usually defined. In particular, the sequence $(\tau_n)_{n\in\N}$ corresponds to the sequence of times at which a new edge appears in the CBP. In this setting, the convergence in probability given in Theorem \ref{the-limitCBP} does not imply the convergence in probability of $(m\e^{-\alpha^* \tau_n}N_k^{\sss(m)}(\tau_n))_{n\in\N}$. Here, we will prove that $\e^{-\alpha^* \tau_n}N^{\sss(m)}_k(\tau_n)$ converges in probability to $p^{\sss(m)}_k\Theta/(\mu\alpha^*)$, and that this further implies that $N^{\sss(m)}_k(\tau_{mn})/n$ converges in probability to $p^{\sss(m)}_k$, as required.

Recall the $t$-bulk sigma-field. We denote, as in \eqref{for-sigmaDEF}, for $n\geq \sub{\xi}^{\I}_t$, 
$$
	\sigma_n=\sigma_n(t)=\frac{1}{\alpha^*} \log{n} -\frac{1}{\mu\alpha^*}\Theta_t.
$$
Take $t=t_n=(\log{n})^{1/2}$. Then, define the sequence $(\tau'_n)_{n\in\N}$, where $\tau'_n := \sigma_n(t_n)$. 
Notice that $\tau'€™_n$ is $t_n$-bulk-measurable. Further, $\tau'_n\stackrel{a.s.}{\rightarrow}\infty$ and 
$$
	\frac{t_n}{\tau'_n} =
		\frac{(\log n)^{1/2}}{\frac{1}{\alpha^*}\log n -\frac{1}{\mu\alpha^*}\log \Theta_{t_n}} = \frac{(\log n)^{1/2}}{\log n(1/\alpha^*-\log \Theta_{t_n}/(\mu\alpha^* \log n))}\stackrel{a.s.}{\longrightarrow}0. 
$$
By Remark \ref{remark:bulk:function}, Proposition \ref{prop-condmoments} holds for $m\e^{-\alpha^* \tau'_n}N_k^{\sss(m)}(\tau'_n)$, so that $m\e^{-\alpha^* \tau'_n}N_k^{\sss(m)}(\tau'_n)\stackrel{\pr}{\rightarrow}p^{\sss(m)}_k\Theta/(\mu\alpha^*)$. The advantage of the sequence $(\tau'_n)_{n\in\N}$, other than being $t_n$-bulk measurable, is that it is a good approximation of the sequence $(\tau_n)_{n\in\N}$. Indeed,
\eqn{
	|\tau_n-\tau'_n|\leq \left|\tau_n-\frac{1}{\alpha^*}\log n-\frac{1}{\mu\alpha^*}\log\Theta\right|+
			\left|\frac{1}{\mu\alpha^*}\log\Theta-\frac{1}{\mu\alpha^*}\log\Theta_{t_n}\right|,
} 
so that $|\tau_n-\tau'_n|\stackrel{a.s.}{\rightarrow}0$. As a consequence,  also $m\e^{-\alpha^* \tau_n}N_k^{\sss(m)}(\tau_n)\stackrel{\pr}{\rightarrow}p^{\sss(m)}_k\Theta/(\mu\alpha^*)$. 

By Theorem \ref{theo-general}, we further know that  $m\e^{-\alpha^* t}N^{\sss(m)}(t)\stackrel{a.s.}{\rightarrow}\Theta/(\mu\alpha^*)$, so this holds also for $m\e^{-\alpha^* \tau_n}N^{\sss(m)}(\tau_n)$. As a consequence,
\eqn{
	\frac{m\e^{-\alpha^* \tau_n}N_k^{\sss(m)}(\tau_n)}{m\e^{-\alpha^* \tau_n}N^{\sss(m)}(\tau_n)} = \frac{N_k^{\sss(m)}(\tau_n)}{N^{\sss(m)}(\tau_n)} = \frac{m}{n}N_k^{\sss(m)}(\tau_n)\stackrel{\pr}{\longrightarrow}p^{\sss(m)}_k.
}
Consequently, $N_k^{\sss(m)}(\tau_{mn})/n\stackrel{\pr}{\longrightarrow}p^{\sss(m)}_k$. This completes the proof of Theorem \ref{th:PAMdiscrete}.
\qed

\bigskip

\noindent
{\bfseries Acknowledgments.}
This work is supported in part by the Netherlands Organisation for Scientific Research (NWO) through the Gravitation {\sc Networks} grant 024.002.003. The work of RvdH is further supported by the Netherlands Organisation for Scientific Research (NWO) through VICI grant 639.033.806.

{\footnotesize \printbibliography[title=References, heading = bibintoc]}
\end{document}